\documentclass{amsart}
\usepackage[utf8]{inputenc}
\usepackage{amsfonts}
\usepackage{amsmath}
\usepackage{mathtools}
\usepackage{amsthm, mathrsfs}
\usepackage{amssymb,color}
\usepackage{latexsym}
\usepackage{enumerate}
\usepackage{tikz-cd} 
\usepackage{enumitem}
\usepackage{hyperref}
\usepackage{rotating}
\usepackage{geometry}

\newtheorem{theorem}{Theorem}[section]
\newtheorem{lemma}[theorem]{Lemma}
\newtheorem{proposition}[theorem]{Proposition}
\newtheorem{corollary}[theorem]{Corollary}

\numberwithin{equation}{section}

\theoremstyle{definition}
\newtheorem{definition}[theorem]{Definition}

\theoremstyle{remark}

\newcommand{\A}{\mathcal{A}}

\newcommand{\B}{\mathcal{B}}

\newcommand{\Q}{\mathcal{Q}}

\newcommand{\ZZ}{\mathbb{Z}}
\newcommand{\QQ}{\mathbb{Q}}
\newcommand{\vphi}{\varphi}

\newcommand{\la}{\langle}
\newcommand{\ra}{\rangle}

\newcommand{\Aff}{\operatorname{Aff}}

\title{A class of AF-algebras up to universal UHF-algebra stability}
\author{Saeed Ghasemi}
\address{Institute of Mathematics, Czech Academy of Sciences, Czech Republic}
\email{\texttt{ghasemi@math.cas.cz}}
\thanks{This research is supported by the GA\v CR project 19-05271Y and RVO: 67985840}

\begin{document}

\maketitle
\begin{abstract}
    We will show that separable unital AF-algebras whose Bratteli diagrams do not allow converging two nodes into one node, can be classified up to the tensor product with the universal UHF-algebra $\Q$ only by their trace spaces. That is, if $\A$ and $\B$ are such AF-algebras, then $T(A)=T(B)$  if and only if $\A\otimes \Q \cong \B\otimes \Q$.
\end{abstract}

\null

\section{Introduction}
UHF-algebras were first studied and classified by Glimm \cite{Glimm}. To any UHF-algebra one assigns a unique a ``supernatural number" (and vise versa) which is a complete isomorphism invariant.  More generally, later 
separable AF-algebras were classified by Elliott \cite{Elliott} using their $K_0$-groups. In the unital case, Elliott's theorem assigns to a unital separable AF-algebra  $\A$ a unique ``dimension group with order-unit"  $\la K_0(\A), K_0(\A)^+, [1_\A]_0 \ra$ (and vice versa), as a complete isomorphism invariant. Since then, the aim of Elliott's classification program has been to classify more classes of separable nuclear C*-algebras. The so called ``strongly self-absorbing" C*-algebras play a particularly important role in the classification program. In fact, the classification results are usually obtained up to tensoring with one of the C*-algebras in the short list of known strongly self-absorbing C*-algebras (refer to \cite{Winter-classification} or \cite{Rordam-Stormer}). In the classification of separable nuclear C*-algebras in addition to the K-theoretic data, the ``trace space" of a C*-algebra, which is tightly related to the state space of the $K_0$-group, is also usually used as an invariant.  In these notes we show that a fairly rich class of AF-algebras can be classified up to tensoring with the ``universal UHF-algebra" $\Q$ (up to $\Q$-stability) using only the trace space.  The universal UHF-algebra is strongly self-absorbing and is the UHF-algebras with $K_0$-group isomorphic to the additional group of all rational numbers $\QQ$ with the usual ordering and 1 as order-unit.

The AF-algebras that we consider are unital and their Bratteli diagrams do not allow edges of the form
$$
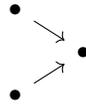
\begin{tikzcd}[column sep=small, row sep=tiny]
\bullet \arrow[dr]& \\
& \bullet \\
\bullet \arrow[ru]&
\end{tikzcd}
$$
Equivalently, these separable unital AF-algebras are precisely the ones that arise as the limit of a direct sequence of finite-dimensional C*-algebras where the connecting maps have exactly one non-zero entry in each row in their representing matrices.
We will denote this class of unital AF-algebras by
 $\overrightarrow{\mathfrak D}$. It contains all the UHF-algebras (for which Theorem \ref{main-thm} is a tautology), all the commutative C*-algebras $C(X)$  for compact, second countable and totally disconnected spaces $X$ (because every such $X$ is the inverse limit of an inverse sequence of finite sets and  surjective maps). Also  $\overrightarrow{\mathfrak D}$ is closed under direct sums and tensor products.
 
For a unital C*-algebra $\A$,  the space of all tracial states of $\A$ is denoted by $T(\A)$. Since $\Q$ has a unique tracial state, it is known and easy to check that $T(\A)\cong T(\A \otimes \Q)$ for any unital C*-algebra $\A$. In particular, for unital AF-algebras $\A, \B$, if $\A \otimes \mathcal Q\cong \B\otimes \mathcal Q$ then we always have $T(\A)\cong T(\B)$ (i.e. they are affinely homeomorphic simplices).
\begin{theorem}\label{main-thm}
Suppose $\A$ and $\B$ are unital AF-algebras in  $\overrightarrow{\mathfrak D}$. Then $T(A)\cong T(B)$ if and only if $\A \otimes \mathcal Q\cong \B\otimes \mathcal Q$.
\end{theorem}
For any unital AF-algebra $\A$, its trace space $T(\A)$ is affinely homeomorphic to ``state space" of the dimension group with order-unit $\la K_0(\A), K_0(\A)^+, [1_\A]_0 \ra$ via the map which sends $\tau$ to $K_0(\tau)$ (see \cite[Proposition 1.5.5]{Rordam-Stormer}).  Therefore Theorem \ref{main-thm} (as well as the whole theory of unital AF-algebras) can be restated in the language of dimension groups with order-units and their state spaces (Theorem \ref{thm-D}). 

\section*{Preliminaries}
An (partially) ordered abelian group $G$ with a fixed order-unit $u$ is denoted by $\la G, u \ra$.
Suppose $G_1, G_2, \dots G_k$ are  ordered abelian groups. The tensor product  $G= G_1\otimes G_2 \otimes \dots\otimes G_k$, as defined in \cite{G-H}, is an ordered abelian group with the positive cone $G^+$ defined as the collection of all finite sums of the elements of the set 
$$
\{x_1\otimes \dots \otimes x_k: x_i\in G_i^+\}.
$$

If $u_i$ is an order-unit of $G_i$ then $u_1\otimes \dots \otimes u_k$ is an order-unit of $G$ (\cite[Lemma 2.4]{G-H}). Tensor products of positive homomorphisms are positive homomorphisms and if $G$ and $H$ are respective direct limits of sequences $(G_i, \alpha_i^j)$ and $(H_i, \beta_i^j)$ of ordered abelian groups, then $G\otimes H$ is the direct limit of the sequence $(G_i\otimes H_i, \alpha_i^j\otimes \beta_i^j)$ (\cite[Lemma 2.2]{G-H}).

  We always consider the abelian group of rational numbers $\QQ$ with its usual ordering and the order-unit 1. In the following write simply $\QQ$ instead of $\la \QQ, 1 \ra$. There is a one-to-one correspondence between the subgroups of $\QQ$ which contain 1 and the supernatural numbers (see \cite[Proposition 7.4.3]{Rordam-K-theory}). For a supernatural number $n$ we denote the corresponding subgroup of $\QQ$ by $\QQ_n$. Multiplication of supernatural numbers is defined naturally as an extension of the multiplication of the natural numbers. If $m$ is also a supernatural number then $\QQ_n\otimes \QQ_m\cong \QQ_{nm}$, i.e., they are isomorphic as ordered groups with order-units. In particular, $\mathbb Q_n \otimes \QQ\cong \QQ$ for every supernatural number $n$.

 Recall that an ordered abelian group is archimedean if $nx\geq y$ for every $n\in \mathbb N$ implies that $x\leq 0$. 
 \begin{proposition} \label{arch}
 Suppose $G$ is an ordered abelian group and $n$ is a supernatural number. Then $G$ is archimedean if and only if $G\otimes \QQ_n$ is archimedean.
 \end{proposition}
\begin{proof}
The map $x \to x\otimes 1$ is a positive embedding from $G$ into $G\otimes \QQ_n$. Therefore if $G$ is not archimedean it is clear that $G\otimes \QQ_n$ is not archimedean. Assume $G$ is archimedean.
Suppose $x,y\in G\otimes \QQ_n$ and $nx\leq y$ for every natural number $n$.  Suppose $x=\sum_{i=1}^r x_i\otimes \frac{a_i}{b_i}$ and $y=\sum_{j=1}^s y_j\otimes \frac{c_j}{d_j}$. By changing the sign of the integers $a_i$ and $c_j$, if necessary, we may assume $b_i$ and $d_j$ are positive natural numbers for every $i\leq r$ and $j\leq s$. Let $b=b_1 \dots b_r>0$, $d= d_1 \dots d_s>0$, $k_i= a_ib/b_i\in \mathbb N$ and $l_j=c_j d/d_j\in \mathbb N$.  Then we have $x=(\sum_i k_i x_i)\otimes \frac{1}{b}$ and $y=(\sum_j l_j y_j)\otimes \frac{1}{d}$. For every $n$
$$
y-nx = (\sum_j l_j y_j)\otimes \frac{1}{d} - n (\sum_i k_i x_i)\otimes \frac{1}{b} = (b\sum_j l_j y_j - nd \sum_i k_i x_i) \otimes \frac{1}{bd} \geq 0.
$$
Since $\frac{1}{bd}>0$ we must have $b\sum_j l_j y_j - nd \sum_i k_i x_i\geq 0$ for every $n$. Since $G$ is archimedean this implies that $d\sum_i k_i x_i \leq 0$. As $d$ is a positive integer, we have $\sum_i k_i x_i \leq 0$ which implies that $x=(\sum_i k_i x_i)\otimes \frac{1}{b}\leq 0$.
\end{proof}

For any non-negative integer $r$, the ordered abelian group $\ZZ^r$ equipped with the positive cone 
$$(\mathbb Z^r)^+ = \{(x_1, x_2, \dots,x_r)\in \mathbb Z^r: x_i\geq 0\}.$$
  is called a \emph{simplicial group}. 
A (not necessarily countable) partially ordered abelian group $\la G, G^+ \ra$ is called a \emph{dimension group} if it is directed, unperforated interpolation
group. We refer the reader to \cite{Goodearl-book} for definitions and more on dimension groups. By a well-known result of Effros,
Handelman and Shen \cite{EHS} any dimension group (with order-unit) is isomorphic to a direct limit of a direct system of simplicial groups (with order-units) and positive (order-unit preserving) homomorphisms, in the category of ordered abelian groups (with order-units).
In particular, a countable dimension group $\la G, G^+, u \ra$ with order-unit $u$ (we would simply write $\la G, u \ra$ in the following) is isomorphic to the direct limit of a sequence simplicial groups and normalized (order-unit preserving)  positive homomorphisms
$$\la\mathbb Z^{r_1} , u_1 \ra \xrightarrow{\alpha_1^2}\la\mathbb Z^{r_2}, u_2\ra \xrightarrow{\alpha_2^3}\la\mathbb Z^{r_3}, u_3 \ra \xrightarrow{\alpha_3^4} ~~\dots ~~\la G, u \ra. $$

 Tensor product of two dimension groups with order-units is again a dimension group with order-unit. In fact, their tensor products correspond to the tensor products of the corresponding AF-algebras. That is, if $\A$ and $\B$ are AF-algebras, then $\la K_0(\A\otimes \B), [1]_{\A\otimes \B}\ra \cong \la K_0(\A), [1]_\A\ra \otimes \la K_0(\B), [1]_\B\ra $ (\cite[Proposition 3.4]{G-H}).

We are only concerned here with countable dimension groups and therefore by a dimension group we mean a countable one.
Clearly $\QQ$ and its subgroups (with the inherited ordering) are dimension groups. 
Every ordered subgroup of $\QQ$ containing 1 is isomorphic as an ordered abelian group with order-unit to the limit of a sequence of simplicial groups and normalized positive homomorphisms  $\la\ZZ, n_1 \ra \to \la \ZZ, n_2 \ra \to \dots$ for some sequence $\{n_i\}$ of natural numbers, such that $n_i|n_{i+1}$ for each $i$. 
We say  such a sequence $\{n_i\}$ of natural numbers is \emph{associated} to the supernatural number $n$ if the limit of the directed sequence $\la\ZZ, n_1 \ra \to \la \ZZ, n_2 \ra \to \dots$ is isomorphic to $\QQ_n$. Note that this is different from the usual unique representation of $n$ as extended powers of prime numbers, since it is not even uniquely associated to $n$. However, it is convenient to use for our purpose. For a natural number $m$ and $x=(x_1, x_2, \dots, x_r)\in \ZZ^r$ we write $mx$ for $(mx_1, mx_2, \dots, mx_r)$.
If $\la G,u \ra =\varinjlim (\ZZ^{r_i}, u_i, \alpha_i^j)$ is a dimension group with order-unit and $\{n_i\}$ is any sequence associated to the supernatural number $n$, then $\la G,u\ra \otimes \QQ_n =\varinjlim (\ZZ^{r_i},n_iu_i,  \frac{n_j}{n_i}\alpha_i^j)$ (cf. \cite[Lemma 2.2]{G-H}).

\begin{definition}\label{eq-def}
Define the equivalence relation $\sim_{\QQ}$ on the set of all  dimension groups with order-units by 
$$
\la G, u\ra \sim_{\QQ} \la H,v \ra \quad \Leftrightarrow \quad \la G, u\ra \otimes \mathbb Q\cong \la H,v \ra \otimes \mathbb Q.
$$
\end{definition}
Equivalently, $\la G, u\ra \sim_{\QQ} \la H,v \ra$ if and only if there are supernatural numbers $n,m$ such that $\la G, u\ra \otimes \mathbb Q_m\cong \la H,v \ra \otimes \mathbb Q_n$ as ordered abelian groups with order-units.

\section{Dimension groups with positive non-mixing connecting maps}

We say a homomorphism $\alpha: \ZZ^r \to \ZZ^s$ is \emph{non-mixing} if  $\alpha(x_1, \dots, x_r) = (k_1x_{i_1}, \dots, k_sx_{i_s})$ where $i_j\in\{1, \dots r\}$ for each $j\leq s$. A non-mixing homomorphism is positive if and only if each $k_j$ is a positive integer. 
Let $\overrightarrow{\mathfrak D}$ denote the class of all dimension groups $G$ that arise as the limit of a direct sequence of simplicial groups and positive non-mixing homomorphisms $(G_i, \alpha_i^j)$.  
Since positive non-mixing homomorphisms send order-units to order-units, the image of each order-unit of each $G_i$ in $G$ is an order-unit of $G$. Note that positive non-mixing homomorphisms are not necessarily embeddings.  A dimension groups that is inductive limit of a sequence of simplicial groups with injective connecting maps is called \emph{ultrasimplicial} \cite{Elliott-t-o-g}. Not all dimension groups are ultrasimplicial \cite[Example 2.7]{Elliott-t-o-g}.
\begin{proposition}
Suppose $G$ is a dimension group in $\overrightarrow{\mathfrak D}$ 
\begin{enumerate}
    \item $G$ is ultrasimplicial.
    \item $G$ is archimedean.
\end{enumerate}
\end{proposition}\label{ultra-arch}
\begin{proof} (1)  Suppose $G$ is the limit of a sequence $(G_i, \alpha_i^j)$ of simplicial groups and positive non-mixing homomorphisms. Restricting each $\alpha_i^{i+1}$ to $G_i/\text{ker}(\alpha_i^{i+1})$ and projecting its image to $G_{i+1}/\text{ker}(\alpha_{i+1}^{i+2})$ yields a sequence with injective maps with the same limit $G$.

(2) By (1) we can assume $\alpha_i^j$ are injective positive non-mixing homomorphisms.
If $z\in G_i$ and $z\not \leq 0$ then $\alpha_i^j(z)\not \leq 0$ for any $j\geq i$ and therefore $\alpha_i^\infty(z)\not \leq 0$. Suppose $x,y\in G$ are such that $nx\leq y$ for every $n\in \mathbb N$, and $x_i, y_i$ are such that $\alpha_i^\infty(x_i)=x$ and $\alpha_i^\infty(y_i)=y$. Then $nx- y\leq 0$ implies that $nx_i-y_i\leq 0$ for every $n$. Since $G_i$ is archimedean we have $x_i \leq 0$ and therefore $x\leq 0$. 
\end{proof}
Next we will show that for dimension groups in $\overrightarrow{\mathfrak D}$ changing the order-unit results in $\sim_\QQ$-equivalent (Definition \ref{eq-def}) dimension groups. First we need the following fairly trivial lemma.
\begin{lemma}\label{lemma1}
Suppose $\alpha: \ZZ^{r} \to \ZZ^{s}$ is a   positive non-mixing homomorphism of simplicial groups and  $\gamma: \ZZ^{r} \to  \ZZ^r$ is a homomorphism defined by $\gamma(x_1, \dots,x_r)= (l_1x_1, \dots, l_rx_r)$  for positive integers $l_1,\dots l_r$. Then there are natural numbers $n$,  positive integers $l'_1,\dots, l'_s$ such that if $\eta: \ZZ^{s} \to \ZZ^{s} $ is defined by $\eta(x_1, \dots,x_s)= (l'_1x_1, \dots, l'_sx_s)$, then $\eta\circ \alpha\circ\gamma= n\alpha$, i.e., the following diagram commutes.
 \begin{equation*} \label{diag1}
\begin{tikzcd}
 \ZZ^r \arrow[d, "\gamma"]\arrow[r,"n\alpha"]& \ZZ^s\\
\ZZ^r \arrow[r, "\alpha"] & \ZZ^s \arrow[u, dashed, "\eta"]  
\end{tikzcd}
\end{equation*}
\end{lemma}
\begin{proof}
Suppose $\alpha(x_1, \dots, x_r) = (k_1x_{i_1}, \dots, k_sx_{i_s})$ for $i_j\in\{1, \dots r\}$.
 Let $n = k_1 \dots k_s l_{i_1} \dots l_{i_s}$
  and define  $\eta(x_1, \dots, x_s) = (\frac{n}{k_1l_{i_1}}x_1, \dots, \frac{n}{k_sl_{i_s}}x_s)$. Then $\eta \circ \alpha \circ \gamma = n \alpha$.
\end{proof}

\begin{proposition}

 \label{thm-uv}
If $G\in \overrightarrow{\mathfrak D}$ and $u,w$ are order-units of $G$ then $\la G, u\ra\sim_{\QQ} \la G, w\ra$.
\end{proposition}
\begin{proof}
Suppose $G = \varinjlim (G_i,\alpha_i^j)$ where $(G_i, \alpha_i^j)$ is a direct sequence of simplicial groups and positive non-mixing homomorphisms.  Without loss of generality assume that there are $u_1, w_1\in G_1$ such that $\alpha_1^\infty(u_1) = u$ and $\alpha_1^\infty(w_1) = w$. Set $u_i = \alpha_1^i(u_1)$ and $w_i =\alpha_1^i(w_1)$. Since each $u_i$ and $w_i$ are order-units of $G_i$ (because $\alpha_1^i$ is order-unit preserving) the convex subgroups (ideals) of $G_i$ generated by $u_i$ and $w_i$ are all of $G_i$. Therefore we have $\la G, u\ra = \varinjlim (G_i, u_i, \alpha_i^j)$ and $\la G, w\ra = \varinjlim (G_i, w_i, \alpha_i^j)$; cf. \cite[Corollary 3.18]{Goodearl-book}. 


Suppose $G_1=\ZZ^{r_1}$ and let $m_1 = u_{11} \dots u_{1r_{1}}$ where $u_1=(u_{11}, \dots, u_{1r_{1}})$. We can find $\gamma_1: \la G_1, u_1 \ra \to \la G_1,m_1w_1 \ra$ given by $(x_1 \dots, x_{r_1}) \to (k_1x_1, \dots, k_{r_1}x_{r_1})$ where $k_i = \frac{u_{11}\dots u_{1r_1}}{u_i}w_i$.
Starting wit $\gamma_1$ use Lemma \ref{lemma1} recursively to find sequences of natural numbers $\{n_i\}$ and $\{m_i\}$ and normalized positive homomorphism $\gamma_i$ and $\eta_i$ such that the following diagram commutes 
 \begin{equation*} \label{diag2}
\begin{tikzcd}
\la G_1, u_1 \ra \arrow[r,"n_1\alpha_1^2"]\arrow[d,"\gamma_1"] & \la G_2, n_1u_2 \ra \arrow[r,"\alpha_2^3"]& \la G_3, n_1u_3 \ra \arrow[r,"n_2\alpha_3^4"]\arrow[d,"\gamma_2"]& \la G_4, n_2n_1 u_4 \ra & \dots  \\
\la G_1, m_1w_1 \ra \arrow[r,"\alpha_1^2"]& \la G_2, m_1w_2 \ra \arrow[r,"m_2\alpha_2^3"]\arrow[u, "\eta_1"]& \la G_3, m_2m_1w_3 \ra \arrow[r,"\alpha_3^4"] & \la G_4, m_2m_1 w_4 \arrow[u, "\eta_2"] \ra& \dots
\end{tikzcd}
\end{equation*}
This intertwining provides an isomorphism between $\la G,u \ra \otimes \QQ_n$ and $\la G,w \ra \otimes \QQ_m$, where $n$ and $m$ are the supernatural numbers associated to the sequences $\{n_i\dots n_1\}_{i\in \mathbb N}$ and $\{m_i\dots m_1\}_{i\in \mathbb N}$, respectively.
\end{proof}

\begin{corollary} \label{coro-wz}
Suppose $G,H \in \overrightarrow{\mathfrak{D}}$ and $\la G, w\ra$ and $\la H, z\ra$ are isomorphic for some order-units $w,z$. Then for every $u$ order-unit of $G$ and $v$ order-unit of $H$ we have $\la G, u\ra\sim_{\QQ} \la H, v\ra$.
\end{corollary}

\begin{proof}
By Proposition \ref{thm-uv} and our assumption we have
$\la G, u \ra \sim_{\QQ} \la G, w \ra \cong \la H, z \ra \sim_{\QQ} \la H, v \ra$.
\end{proof}

\subsection{State spaces}
Let $\la G,u \ra$ be an ordered abelian group with order-unit $u$.
 The set $S( G,u )$ of all states on $\la G,u \ra$, called the \emph{state space} of $ \la G,u \ra$, when equipped with the relative topology endowed from $\mathbb R^G$, is a compact convex subset of the locally convex topological vector space $\mathbb R^G$. Given a normalized positive homomorphism $\alpha: \la G, u \ra \to \la H, v \ra$, one associates an affine and continuous map $S(\alpha): S(  H,v  ) \to S(  G,u )$ defined by $S(\alpha)(s) =s \circ \alpha$, for every $s\in S( H, v )$. In fact, $S$ is a contravariant and continuous functor from the category of ordered abelian groups with order-unit into the category of compact convex sets.
If $\la G, u \ra$ is a dimension group with an order-unit, then $S( G,u )$ is a  Choquet simplex (\cite[Theorem 10.17]{Goodearl-book}). Conversely, a Choquet simplex is affinely homeomorphic to the state space of a (countable) dimension group with an order-unit (\cite[Corollary 14.9]{Goodearl-book}). 

If $G$ is a ordered abelian group and $u_1, u_2 \in G^+$ are order-units of $G$, then the map $\phi: S(G, u_1) \to S(G, u_2)$ defined by $\phi(s)(x) = \frac{1}{s(u_2)}s(x)$ is a (not necessarily affine) homeomorphism. In fact, there are dimension groups $G$ with order-units $u$ and
$v$ such that $S(G,u)$ is not affinely homeomorphic to $S(G,v)$ (see \cite[Example 6.18]{Goodearl-book}). 
\begin{proposition}\label{prop-S-Q}
For every dimension group with order-unit $\la G, u \ra$ and supernatural number $n$ we have $S(G, u) \cong S(\la G, u\ra \otimes \mathbb Q_n)$. If $\la G ,u \ra \sim_{\QQ} \la H, v \ra$, then $S(G,u)\cong S(H,v)$.
\end{proposition}
\begin{proof}
If $\la G, u \ra = \varinjlim (G_i, u_i, \alpha_i^j)$ where each $\la G_i, u_i\ra$ is a simplicial group with order-unit, then for any supernatural number $n$ and an associated sequence $\{n_i\}$, we have $\la G, u \ra\otimes \mathbb Q_n = \varinjlim (G_i, n_iu_i, \frac{n_j}{n_i}\alpha_i^j)$. The following diagram commutes.
 \begin{equation*} \label{diag2}
\begin{tikzcd}
\la G_1, u_1 \ra \arrow[r,"\alpha_1^2"]\arrow[d,"n_1"] & \la G_2, u_2 \ra \arrow[r,"\alpha_2^3"]\arrow[d,"n_2"]& \la G_3, u_3 \ra \arrow[r,"\alpha_3^4"]\arrow[d,"n_3"] & \dots & \la G,u\ra\\
\la G_1, n_1u_1 \ra \arrow[r,"\frac{n_2}{n_1}\alpha_1^2"]& \la G_2, n_2u_2 \ra \arrow[r,"\frac{n_3}{n_2}\alpha_2^3"]& \la G_3, n_3u_3 \ra \arrow[r,"\frac{n_4}{n_3}\alpha_3^4"] & \dots & \la G, u \ra\otimes \mathbb Q_n
\end{tikzcd}
\end{equation*}
After applying the state functor $S$ to the above diagram we get a commuting diagram where the inverse limit of the first row is $S(G,u)$ and the inverse limit of the second row is $S(\la G, u \ra\otimes \mathbb Q_n)$ (cf. \cite[Proposition 6.14]{Goodearl-book}). Clearly $S(n_i): S(G_i, n_iu_i) \to S(G_i, u_i)$ is an affine homeomorphism and therefore  $S(G, u) \cong S(\la G, u\ra \otimes \mathbb Q_n)$. The second statement follows immediately.
\end{proof}
\begin{corollary}
Suppose $u$ and $v$ are order-units of a dimension group $G\in \overrightarrow{\mathfrak D}$. Then $S(G,u)$ is affinely homeomorphic to $S(G,v)$.
\end{corollary}
\begin{proof}
Follows from Proposition \ref{thm-uv} and Proposition \ref{prop-S-Q}.
\end{proof}
In general $S(G,u) \cong S(H,v)$ does not imply $\la G,u\ra \sim_{\QQ} \la H,v \ra$, not even in the finite-dimensional case. For example, simply let $\la G, u \ra$ denote the simplicial group $\la\ZZ^2, 1\ra$ and $\la H, v \ra$ denote the group $\QQ^2$ with the strict ordering $\ll$ and order-unit 1. Note that $\la H, v \ra$ is a dimension group since it is a directed and unperforated interpolation group. Clearly $S(G,u) \cong S(H,v)\cong \Delta_1$ but $\la G,u\ra \not\sim_{\QQ} \la H,v \ra$, since the left side is archimedean while the right side is not, and tensoring with $\QQ$ does not change the state of being archimedean (Proposition \ref{arch}). However, we will show that for dimension groups $G,H$ in $\overrightarrow{\mathfrak D}$ we have $S(G,u) \cong S(H,v)$ implies that $\la G,u\ra \sim_{\QQ} \la H,v \ra$. 

 Given a compact convex set $K$ in a linear topological space $\text{Aff}(K)$ denotes the collection of all affine continuous real-valued functions on $K$. Let $ 1_K$ denote the constant function with value 1 on $K$. Equipped with the pointwise ordering, $\la \text{Aff}(K), 1_k\ra$ is an ordered real vector space with order-unit. With the supremum norm $\Aff(K)$, as a
norm-closed subspace of $C(K,\mathbb R)$, is an ``ordered real Banach space". A compact convex $K$ is a Choquet simplex if and only if $\Aff(K)$ is a (uncountable) dimension group (cf. \cite[Theorem 11.4]{Goodearl-book}).  

For any affine continuous
map $\nu : K \to L$ between compact convex sets, the map $A(\nu): \la \text{Aff}(L), 1_L\ra \to \la\text{Aff}(K), 1_K \ra$ defined by $A(\nu)(f) = f\circ \nu$, for $f\in \text{Aff}(L)$ is an order-unit preserving positive homomorphism. 

\subsection{The main result}
Suppose $\la G, u \ra$ is an ordered abelian group with order-unit. Then there is a natural normalized positive homomorphism $\vphi: \la G, u \ra \to \la\Aff(S(G,u)), 1 \ra$ defined by $x\to \widehat x$ where $\widehat x(s) = s(x)$ for any state $s$. The map $\vphi$ is called the \emph{natural affine representation} of $\la G,u \ra$ and it is an embedding if and only if $G$ is archimedean (cf. \cite[Theorem 7.7]{Goodearl-book}). 

Note that $\Aff(S(\ZZ^r,u))$ is isomorphic to $\la\mathbb R^r, 1\ra$ as ordered real Banach space. Any normalized positive homomorphism $\beta:\la\mathbb R^r, 1\ra \to \la\mathbb R^s, 1\ra$ is of the form
$$
\beta(x_1, \dots, x_r)= (\sum_i^r a_{i1}x_i, \dots , \sum_i^r a_{is}x_i)
$$
where $a_{ij}\in \mathbb R$ and $\sum_i^r a_{ij}=1$ for every $1 \leq j\leq s$. We say $\beta$ is \emph{rational} if each $a_{ij}$ is a rational number.

\begin{theorem} \label{thm-D}
Suppose $G,H$ are in $\overrightarrow{\mathfrak D}$ and $u\in G$ and $v\in H$ are order-units. Then $S(G,u)\cong S(H,v)$ if and only if $\la G, u \ra \sim_{\QQ}\la H, v \ra $. 
\end{theorem}
\begin{proof}
We only need to prove the forward direction.  First assume $S(G,u)$ and $S(H,v)$ are finite-dimensional simplices.  If $S(G,u)\cong S(H,v)$ then $\Aff(S(G,u))\cong \Aff(S(H,v))\cong \mathbb R^k$ for some $k$. Since $G$ and $H$ are archimedean (Proposition \ref{ultra-arch}), they are isomorphic to countable subgroups of $\mathbb R^k$ via positive homomorphisms which send $u$ and $v$ to $1$. This together with the fact that every countable additive subgroup of $\mathbb R$ is isomorphic to $\ZZ$ or $\QQ$ implies that $\la G, u \ra \otimes \QQ \cong \la H, v \ra \otimes \QQ$.   

` Now assume  $S(G,u)$ and $S(H,v)$ are infinite-dimensional. Suppose $\la G, u \ra = \varinjlim (G_i, u_i, \alpha_i^j)$ and $\la H, v \ra = \varinjlim (H_i, v_i, \beta_i^j)$, where $\la G_i, u_i\ra$ and $\la H_i,v_i \ra$ are simplicial groups. Without loss of generality we can assume that $G_i \cong H_i \cong \ZZ^i$.
Let $A =\la \Aff(S(G,u)), 1 \ra$, $B = \la \Aff(S(H,v)), 1 \ra$, $A_i = \la \Aff(S(G_i,u_i)), 1 \ra$ and $B_i =\la \Aff(S(H_i,v_i)), 1 \ra$. Let $\widetilde \alpha_i^j: A_i \to A_j$ denote the map $A(s(\alpha_i^j))$ and $\widetilde \beta_i^j: B_i \to B_j$ denote the map $A(s(\beta_i^j))$.
Also let $\vphi: \la G,u\ra \to A$, $\psi: \la H,v\ra \to B$,  $\vphi_i: \la G_i,u_i\ra \to A_i$ and $\psi: \la H_i,v_i\ra \to B_i$ be the respective natural affine representations. Since $\la G,u \ra$ and $\la H, v \ra$ are archimedean, the natural maps $\vphi$ and $\psi$ are embeddings and therefore we may identify $\la G, u \ra$ with $\la \widetilde G, 1_A \ra$ and $\la H, v \ra$ with $\la \widetilde H, 1_B \ra$ as ordered subgroups of $A$ and $B$, respectively. We may also identify each $\la G_i,u_i \ra$ with $\la \widetilde G_i, 1_{A_i} \ra$, its image under the map $\vphi_i$ which sends $(x_1,\dots, x_i)\in\ZZ^i$ to $(x_1/u_{i1}, \dots, x_i/u_{ii})\in \QQ^i$, where $u_i = (u_{i1}, \dots , u_{ii})$, as an ordered subgroup of $A_i$ of rational coordinates. It is easy to check that the diagram
 \begin{equation*} \label{diag2}
\begin{tikzcd}
\la G_1, u_1 \ra \arrow[r,"\alpha_1^{2}"]\arrow[d,"\vphi_1"] & \dots &
\la G_i, u_i \ra \arrow[r,"\alpha_i^{i+1}"]\arrow[d,"\vphi_i"] & \la G_{i+1}, u_{i+1} \ra \arrow[d,"\vphi_{i+1}"] \arrow[r]& \dots &  \la G,u \ra \arrow[d, "\vphi"]\\
\la \widetilde G_1, 1_{A_1}\ra \arrow[r,"\widetilde\alpha_1^{2}"]& \dots &
\la \widetilde G_i, 1_{A_i}\ra \arrow[r,"\widetilde\alpha_i^{i+1}"]& \la \widetilde G_{i+1}, 1_{A_{i+1}} \ra  \arrow[r]&  \dots & \la \widetilde G, 1_A \ra
\end{tikzcd}
\end{equation*}
commutes, i.e. $\widetilde\alpha_i^{i+1}(\widehat x) = \widehat{\alpha_i^{i+1}(x)}$  and $\widetilde\alpha_i^{\infty}(\widehat x) = \widehat{\alpha_i^{\infty}(x)}$ for every $i$ and $x\in G_i$  (this is true for any positive homomorphism and not just the non-mixing ones). Therefore 
$$\la G, u \ra \cong \la \widetilde G, 1_A \ra\cong \varinjlim (\widetilde G_i, 1_{A_i}, \widetilde\alpha_i^j) \leqno(*)$$ 
and similarly, $$\la H,v \ra \cong \la \widetilde H, 1_B \ra\cong \varinjlim (\widetilde H_i, 1_{B_i}, \widetilde\beta_i^j), \leqno(**)$$
where $\widetilde H_i$ is a subgroup of $\QQ^i$ in $B_i$. Since $S(G,u)\cong S(H,v)$, the ordered real Banach spaces $\la A, 1 \ra$ and $\la B, 1 \ra$ are isomorphic, which induces an approximate intertwining between the systems $(A_i, 1_{A_i}, \widetilde\alpha_i^j)$ and $(B_i, 1_{B_i}, \widetilde\beta_i^j)$ as follows. Let $X_i$ denote the natural set of generators of $\widetilde G_i$ and $Y_i$ denote the natural set of generators of $\widetilde H_i$ as subgroups of $\QQ^i$ (note that $|X_i| = |Y_i| = i$). Find increasing sequences of natural numbers $\{k_i\}$ and $\{l_i\}$ and normalized rational positive homomorphisms $\gamma_i: A_{k_i} \to B_{l_i}$ and $\eta_i: B_{l_i} \to A_{k_{i+1}}$ and finitely generated subgroups of rational coordinates $E_i\subseteq \QQ^{k_i}\subseteq A_{k_i}$ and $F_i\subseteq \QQ^{l_i}\subseteq B_{k_i}$ such that for every $i$ we have
\begin{itemize}
    \item $X_{k_i}\subseteq E_i$ and $Y_{l_i}\subseteq F_i$,
    \item $\widetilde \alpha_{k_i}^{k_{i+1}}[E_i] \subseteq E_{i+1}$, $\widetilde \beta_{l_i}^{l_{i+1}}[F_i] \subseteq F_{i+1}$, $\gamma_i[E_i]\subseteq F_{i}$ and $\eta_i[F_i]\subseteq E_{i+1}$,
    \item $\|\eta_i \circ\gamma_i(x) - \alpha_{k_i}^{k_{i+1}}(x)\|<1/2^{i}$  and $\|\gamma_{i+1} \circ\eta_i(y) - \beta_{l_i}^{l_{i+1}}(y)\|<1/2^{i}$ for every $x\in E_i$ and $y\in F_i$.
\end{itemize}
Finding such an approximate intertwining between 
$ (E_i, 1_{A_{k_i}}, \widetilde\alpha_i^j)$ and $(F_i, 1_{B_{l_i}}, \widetilde\beta_i^j)$ is an standard argument and it implies that the two sequences have isomorphic limits. Since $E_i$ and $F_i$ are finitely generated and $X_{k_i}\subseteq E_i$ and $Y_{l_i}\subseteq F_i$, we have $\la E_i, 1_{A_{k_i}}\ra \cong \la \ZZ^{k_i}, w_i \ra$ and $\la F_i, 1_{B_{l_i}}\ra \cong \la \ZZ^{l_i}, z_i \ra$ for some order-units $w_i, z_i$. Note that in particular, $E_i \cong \widetilde G_{k_i}$ and $F_i\cong \widetilde H_{l_i}$. Hence by $(*)$ and $(**)$ for some order-units $w\in G$ and $z\in H$ we have 
$$\la G, w \ra = \varinjlim (E_i, 1_{A_{k_i}}, \widetilde\alpha_{k_i}^{k_j}) = \varinjlim (\ZZ^{k_i}, w_i , \alpha_{k_i}^{k_j})$$ 
and
$$\la H, z \ra = \varinjlim (F_i, 1_{B_{l_i}}, \widetilde\beta_{l_i}^{l_j}) = \varinjlim (\ZZ^{l_i}, z_i, \beta_{l_i}^{l_j}).$$ 
Therefore $\la G, w \ra \cong \la H,z \ra$. By Corollary \ref{coro-wz} we have $\la G, u \ra \sim_{\QQ} \la H,v \ra$.
\end{proof}
Let $\overrightarrow{\mathfrak D}$ also denote the class of unital AF-algebras $\A$ such that the dimension group $K_0(\A)$ belongs to $\overrightarrow{\mathfrak D}$.
\begin{corollary}
Suppose $\A$ and $\B$ are unital AF-algebras in $\overrightarrow{\mathfrak D}$. Then $T(A)\cong T(B)$ if and only if $\A \otimes \mathcal Q\cong \B\otimes \mathcal Q$.
\end{corollary}

\bibliographystyle{amsplain}

\end{document}